\documentclass[12pt,oneside]{amsart}
\usepackage{faktor}
\usepackage[hypertexnames=false]{hyperref}
\hypersetup{
	colorlinks=true,
	linktoc=all,
	linkcolor=black,
}

\usepackage[title]{appendix}
\usepackage{csquotes}
\MakeOuterQuote{"}
\usepackage{comment}
\usepackage{mathrsfs}
\usepackage{graphicx}
\usepackage{latexsym,amssymb,amscd}
\usepackage{amsmath}
\usepackage[all]{xy}
\usepackage{amsthm}
\usepackage{mathrsfs}
\usepackage{times,enumerate,scalerel}
\usepackage[usenames,dvipsnames]{color}
\usepackage{blkarray}
\usepackage{amsmath}
\usepackage{faktor}
\usepackage{tikz}

\usetikzlibrary{arrows,decorations.pathmorphing,backgrounds,positioning,fit,matrix}
\usepackage{tikz-cd}
\usepackage{float}
\usepackage{hhline}
\usepackage[english]{babel}

\newcommand{\eqnum}{\refstepcounter{equation}\textup{\tagform@{\theequation}}}
\renewcommand{\boxtimes}{\mathbin{\scalerel*{\tikz{\draw[line width=1.1pt](0,0)rectangle(1,1)--(0,0)(1,0)--(0,1);}}{\otimes}}}

\usepackage{mathtools}
\usepackage{stackengine}
\setstackEOL{\\}
\usepackage{tikz-cd}
\usetikzlibrary{arrows,decorations.pathmorphing,backgrounds,positioning,fit,petri,calc,shapes.misc,decorations.markings}
\tikzset{degil/.style={
		decoration={markings,
			mark= at position 0.5 with {
				\node[transform shape] (tempnode) {$\backslash$};
			}
		},
		postaction={decorate}
	}
}

\usepackage{enumitem}

      \usepackage{amssymb}

      \theoremstyle{plain}
      \newtheorem{theorem}{Theorem}[section]
      \newtheorem{lemma}[theorem]{Lemma}
      
      \newtheorem{proposition}[theorem]{Propositition}
      
      \newtheorem{fact}[theorem]{Fact}
      
      \theoremstyle{definition}

      \theoremstyle{remark}
      \newtheorem{remark}[theorem]{Remark}

      \makeatletter
      \def\@setcopyright{}
      \def\serieslogo@{}
      \makeatother

      \let\OLDthebibliography\thebibliography
      \renewcommand\thebibliography[1]{
      	\OLDthebibliography{#1}
      	\setlength{\parskip}{8pt}
      	\setlength{\itemsep}{0pt plus 0.3ex}
      }
   
\begin{document}

%



   \author{Spyridon Afentoulidis-Almpanis}
   
    \address{Dep. of Mathematics,
    	Bar-Ilan University,
    	Ramat-Gan, 5290002 Israel}

   \email{spyridon.almpanis@biu.ac.il}

   \author{Gang Liu}

    \address{Universit\' e de Lorraine, CNRS, IECL, F-57000 Metz, France}
 
   \email{gang.liu@univ-lorraine.fr}
   

   \title[Tensor product decomposition for rank-one spin groups]{Tensor product decomposition for rank-one spin groups  I : unitary principal series representations}


 \begin{abstract}
	We provide an explicit direct integral decomposition for the tensor product representation $\pi_1\widehat{\otimes}\pi_2$ of the rank one spin group $\mathrm{Spin}(n,1)$ whenever $\pi_1$ is a unitary principal series representation and $\pi_2$ is an arbitrary irreducible unitary representation of $\mathrm{Spin}(n,1)$. 
\end{abstract}



   \thanks{The first named author was supported by the Israel Science
   	Foundation (grant No. 1040/22). The 
   	second named author was partially supported by ANR-23-CE40-0016.}




   \maketitle



\section{Introduction}

	Suppose that $G$ is a real reductive group and $\pi_1$ and $\pi_2$ are irreducible unitary representations of $G$. Since $G$ is of Type I (in the sense of von Neumann algebras), the tensor product representation $\pi_1\widehat{\otimes}\pi_2$ of $G$ has an essentialy unique direct integral decomposition \cite{DixmierVonN}. The tensor product decomposition problem consists in describing this direct integral decomposition. It is of great importance in representation theory and abstract harmonic analysis as well as in theoretical physics, where tensor product decompositions are used to describe fusion rules in quantum mechanics.
	
	It is worthwhile to mention that the tensor product decomposition problem is a special case of the more general branching problems. More precisely, the decomposition of $\pi_1\widehat{\otimes}\pi_2$ is exactly the branching law of the exterior tensor product representation $\pi_1\widehat{\boxtimes}\pi_2$ of $G\times G$ with respect to the diagonally embedded subgroup $G$ of $G\times G$. 
	Branching problems have been extensively studied during the last 40 years, notably by T. Kobayashi and his collaborators \cite{koba1,koba2,koba3}. For more details regarding recent developments of the topic, we refer the interested reader to Kobayashi's Sugaku exposition article \cite{kobayashiRecent}.

In the case of compact Lie groups, where all unitary irreducible representations are finite-dimensional, the tensor product decomposition is known (see \cite{kumar} and references herein) while softwares that compute explicitly this decomposition have been successfully developed (e.g., \cite{Liesoft,atlas}). Nevertheless, for an arbitrary noncompact reductive group, and despite the progress made in the last 50 years (e.g., \cite{repka,martinSL,martin, vandijk,kobaShi,kobaLieTheory,benoist23}), few are, in general, known. 

Let us say a few words about some already known results concerning the tensor product decomposition problem. In \cite{repka}, Repka describes the decomposition of $\pi_1\widehat{\otimes}\pi_2$ for any unitary irreducible representations $\pi_1$ and $\pi_2$ of $\mathrm{SL}(2,\mathbb{R})$. In \cite{martinSL}, by using Mackey's theory, Martin finds a connection between the decomposition of $\pi_1\widehat{\otimes}\pi_2$ for $\mathrm{SL}(2,k)$ and restrictions of unitary irreducible representations of $\mathrm{SL}(2,k)$ to the Borel subgroup whenever $k=\mathbb{R}$ or $\mathbb{C}$. In \cite{martin}, Martin adapts his methods for $\mathrm{Spin}(4,1)$ and provides the decomposition of $\pi_1\widehat{\otimes}\pi_2$ whenever $\pi_1$ is a unitary principal series representation and $\pi_2$ is any nontrivial unitary irreducible representation of $\mathrm{Spin}(4,1)$. In \cite{vandijk}, Van Dijk and Molchanov describe the tensor product decomposition $\pi^+\widehat{\otimes}\pi^-$ for opposite maximal degenerate series representations $\pi^+$ and $\pi^-$ of $\mathrm{SL}(n,\mathbb{R})$.

In this article, we study the tensor product decomposition  for the rank one spin groups $\mathrm{Spin}(n, 1)$. More precisely, we provide an explicit decomposition for $\pi_1\widehat{\otimes}\pi_2$ whenever $\pi_1$ is a unitary principal series representation and $\pi_2$ is an arbitrary irreducible unitary representation of $\mathrm{Spin}(n, 1)$.

The groups $\mathrm{Spin}(n, 1)$ are interesting for several reasons. Firstly, according to the work of Kobayashi \cite{kobaShi}, for a real noncompact simple Lie group $G$, the multiplicities in $\pi_1\widehat{\otimes} \pi_2$ are finite for any irreducible smooth representations (of moderate growth) $\pi_1$ and $\pi_2$ of $G$ if and only if the Lie algebra of $G$ is $\mathfrak{so}(n, 1)$. Since the multiplicity for the direct integral of irreducible unitary representations is bounded by the multiplicity in the smooth category almost everywhere, the multiplicities in the tensor product representations $\pi_1\widehat{\otimes} \pi_2$  are finite for any irreducible unitary representations $\pi_1$ and $\pi_2$ of $\mathrm{Spin}(n, 1)$. Thus, in this sense, the tensor products of irreducible unitary representations of $\mathrm{Spin}(n, 1)$ behave nicely. Secondly, the groups $\mathrm{Spin}(n, 1)$ are also important in physics; these groups and their representations have been extensively studied and used in quantum physics (e.g., \cite{walkschaub}). Therefore, mathematical results concerning the tensor product decomposition for $\mathrm{Spin}(n, 1)$ can also be potentially useful for theoretical physics.

Now let us give a brief description of the methods we use to obtain the decomposition in the case that we are interested in. 
Let $G=\mathrm{Spin}(n, 1)$ and $P$ be a minimal parabolic subgroup of $G$. Let $\pi_1$ be a unitary principal series representation and $\pi_2$ be any nontrivial irreducible unitary representation of $G$.
First, we use an observation by Martin for the group $\mathrm{Spin}(4, 1)$ \cite{martin}, which reveals a certain relation between tensor products and branching laws with respect to parabolic subgroups. We adapt this observation for general $\mathrm{Spin}(n,1)$. Then based on the explicit decomposition of $\pi_2\vert_{P}$ obtained in \cite{gangoshima} and well-known results for finite-dimensional representations, we can prove that as a representation of $G$, the tensor product  $\pi_1\widehat{\otimes}\pi_2$ is isomorphic to a finite sum of certain induced representations "$\mathrm{Ind}_{P}^{G}(\Omega_{jkt})$". Thus, in order to decompose the tensor product  $\pi_1\widehat{\otimes}\pi_2$, it suffices to decompose each induced representation $\mathrm{Ind}_{P}^{G}(\Omega_{jkt})$. It turns out that all $\Omega_{jkt}$ are discrete series representations of $P$. This fact enables us to apply Anh-Mackey's Frobenius reciprocity to decompose explicitly $\mathrm{Ind}_{P}^{G}(\Omega_{jkt})$  based on the Plancherel formula for $G$. 

This article is the first paper of our project concerning the tensor product decomposition for $\mathrm{Spin}(n, 1)$. If neither $\pi_1$ nor $\pi_2$ is a unitary principal series representation, then the methods utilised in this paper cannot apply directly. However, some techniques used in this  article can still be adapted in some "nonunitary framework" especially for complementary series representations. This is an ongoing investigation, which we hope to finish soon.

The article is organized as follows: In Section \ref{preliminaries}, we present the necessary structure theory of $\mathrm{Spin}(n,1)$ and its (minimal) parabolic subgroup $P$ as well as the basic representation theory of $P$. In Section \ref{decomposition1}, we develop our methods and proceed with the tensor product decomposition of $\pi_1\widehat{\otimes}\pi_2$ whenever $\pi_1$ is a (unitary) principal series representation and $\pi_2$ is any nontrivial unitary irreducible representation of $\mathrm{Spin}(n,1)$. The main results of this article are contained in Theorem \ref{mainresult}. For the benefit of the reader, a short appendix with results from \cite{gangoshima} for branching rules of complementary series representations of $\mathrm{Spin}(n,1)$ to $P$ has been added at the end of the article.

\vspace{2mm}
\textbf{Acknowledgements:} The authors would like to thank T. Kobayashi for his useful comments and suggestions on a preliminary version of this article.

\section{Preliminaries}\label{preliminaries}
Let $m$ be a positive integer greater or equal to $3$ and set 
\begin{equation*}
	I_{m+1,1}\coloneqq
	\begin{pmatrix}
		I_{m+1}&\\
		&-1
		\end{pmatrix}
\end{equation*}
where $I_{m+1}$ is the $(m+1)$-by-$(m+1)$ identity matrix. 
Let $G$ be the (nontrivial) double cover $\mathrm{Spin}(m+1,1)$ of the group 

\begin{equation*}
	\mathrm{SO}(m+1,1)\coloneqq\{X\in M_{m+2}(\mathbb{R})\mid XI_{m+1,1}X^{t}=I_{m+1,1}\}
\end{equation*}
so that its Lie algebra $\mathfrak{g}$ is 
\begin{equation*}
	\mathfrak{g}=\mathfrak{so}(m+1,1)\coloneqq\{X\in \mathfrak{gl}(m+2,\mathbb{R})\mid XI_{m+1,1}+I_{m+1,1}X^t=0\}.
\end{equation*}
More details and an explicit construction of spin groups can be found in \cite{goodman}.
If 
\begin{equation*}
	\mathfrak{k}\coloneqq\{\mathrm{diag}(Y,0)\mid Y\in\mathfrak{so}(m+1)\}
\end{equation*}
and 
\begin{equation*}
	\mathfrak{p}\coloneqq\Big\{
	\begin{pmatrix}
		0_{(m+1)\times (m+1)}& \alpha^t\\
		\alpha&0
	\end{pmatrix}
	\in\mathfrak{g}\mid \alpha \in M_{1\times(m+1)}(\mathbb{R})
	\Big\}
	\end{equation*}
then $\mathfrak{g}=\mathfrak{k}\oplus\mathfrak{p}$ is a Cartan decomposition coming from the Cartan involution $\theta\coloneqq \mathrm{Ad}(I_{m+1,1})$.
The element 
\begin{equation*}
	H_0\coloneqq 
	\begin{pmatrix}
		0_{m\times m}& 0_{m\times1}& 0_{m\times1}\\
		0_{1\times m}& 0&1\\
		0_{1\times m}&1&0
		\end{pmatrix}
\end{equation*}
generates a maximal abelian subspace $\mathfrak{a}$ in $\mathfrak{p}$ and let $\mathfrak{m}\coloneqq Z_{\mathfrak{k}}(\mathfrak{a})$. The set of restricted roots $\Sigma(\mathfrak{g},\mathfrak{a})$ consists of $\pm\lambda_0$ where $\lambda_0\in \mathfrak{a}^*$ is defined by $\lambda_0(H_0)=1$ and let $\lambda_0$ be the positive restricted root. The corresponding positive and negative nilradicals $\mathfrak{n}$ and $\bar{\mathfrak{n}}$ are, then, 
\begin{equation*}
	\mathfrak{n}=\big\{
	\begin{pmatrix}
		0_{m\times m}&-\alpha^t&\alpha^t\\
		\alpha& 0&0\\
		\alpha&0&0
	\end{pmatrix} \mid \alpha\in M_{1\times m}(\mathbb{R})    \big\}
\end{equation*}
and 
\begin{equation*}
	\bar{\mathfrak{n}}=\big\{
	\begin{pmatrix}
		0_{m\times m}&\alpha^t&\alpha^t\\
		-\alpha& 0&0\\
		\alpha&0&0
	\end{pmatrix} \mid \alpha\in M_{1\times m}(\mathbb{R})    \big\}
\end{equation*}
respectively. The element
\begin{equation*}
	X_{\alpha}\coloneqq 
	\begin{pmatrix}
			0_{m\times m}&-\alpha^t&\alpha^t\\
		\alpha& 0&0\\
		\alpha&0&0
	\end{pmatrix},
\end{equation*}
where $\alpha\in M_{1\times m}(\mathbb{R})$, is a positive  restricted root vector of $\lambda_0$ while the element 
\begin{equation*}
	\bar{X}_{\alpha}\coloneqq 
	\begin{pmatrix}
		0_{m\times m}&\alpha^t&\alpha^t\\
		-\alpha& 0&0\\
		\alpha&0&0
	\end{pmatrix},
\end{equation*}
is a negative restricted root vector of $-\lambda_0$.
The element
\begin{equation*}
	\rho'\coloneqq \frac{m}{2}\lambda_0\in\mathfrak{a}^*
\end{equation*}
is half the sum of the positive restricted roots of $\Sigma(\mathfrak{g},\mathfrak{a})$.

Let $n'\coloneqq [\frac{m+1}{2}]$. For every $a=(a_1,\ldots,a_{n'})\in\mathbb{R}^{n'}$, set 
\begin{equation*}
	t_a\coloneqq 
	\begin{pmatrix}
		0&a_1&&&&\\
		-a_1&0&&&&\\
		&&\ddots&&&\\
		&&&0&a_{n'}&\\
		&&&-a_{n'}&0&\\
		&&&&&0{(m+2-2n')\times (m+2-2n')}
	\end{pmatrix}
\end{equation*}
and $\mathfrak{t}\coloneqq \{t_a\mid a\in \mathbb{R}^{n'}\}$. Then $\mathfrak{t}$ is a Cartan subalgebra of $\mathfrak{k}$ and  $\mathfrak{t}_s\coloneqq (\mathfrak{t}\cap\mathfrak{m})\oplus\mathfrak{a}$ is a Cartan subalgebra of $\mathfrak{g}$. 

The root system $\Delta(\mathfrak{k}_\mathbb{C},\mathfrak{t}_\mathbb{C})$ of $\mathfrak{k}_\mathbb{C}$ with respect to $\mathfrak{t}_\mathbb{C}$ is 
\begin{equation*}
	\Delta(\mathfrak{k}_\mathbb{C},\mathfrak{t}_\mathbb{C})\coloneqq \{\pm\varepsilon'_i\pm\varepsilon'_j, \pm\varepsilon'_k\mid 1\leq i<j\leq n', 1\leq k\leq n' \}
\end{equation*}
if $m$ is even, and 
\begin{equation*}
	\Delta(\mathfrak{k}_\mathbb{C},\mathfrak{t}_\mathbb{C})\coloneqq \{\pm\varepsilon'_i\pm\varepsilon'_j\mid 1\leq i<j\leq n\}
\end{equation*}
if $m$ is odd, where 
$\varepsilon_i'\in \mathfrak{t}_\mathbb{C}^*$ is defined by 
\begin{equation*}
	\varepsilon'_j: t_a\mapsto ia_j.
\end{equation*}

The root system of $\mathfrak{g}_\mathbb{C}$ with respect to $(\mathfrak{t}_s)_{\mathbb{C}}$ is 
\begin{equation*}
	\Delta(\mathfrak{g}_\mathbb{C},(\mathfrak{t}_s)_\mathbb{C})\coloneqq \{\pm\varepsilon_i\pm\varepsilon_j\mid 1\leq i<j\leq n\},
\end{equation*}
if $m$ is even, and 
\begin{equation*}
	\Delta(\mathfrak{g}_\mathbb{C},(\mathfrak{t}_s)_\mathbb{C})\coloneqq \{\pm\varepsilon_i\pm\varepsilon_j, \pm\varepsilon_k\mid 1\leq i<j\leq n, 1\leq k\leq n \},
\end{equation*}
if $m$ is odd,
where $n\coloneqq [\frac{m+2}{2}]$ and $\varepsilon_i$ is defined by 
\begin{equation*}
	\varepsilon_i=\varepsilon_i' \text{ on } \mathfrak{t}\cap \mathfrak{m} \text{ and } \varepsilon_i=0 \text{ on } \mathfrak{a}
\end{equation*}
if $1\leq i <n$, and 
\begin{equation*}
	\varepsilon_n=0 \text{ on } \mathfrak{t}\cap \mathfrak{m} \text{ and } \varepsilon_n=\lambda_0 \text{ on } \mathfrak{a}.
\end{equation*}
An element $\gamma$ of $(\mathfrak{t}_s)_{\mathbb{C}}^*$ can be expressed as a sum
\begin{equation*}
	\gamma=c_1\varepsilon_1+\ldots+c_n\varepsilon_n=:(c_1,\ldots, c_n)
\end{equation*}
of the elements $\varepsilon_i$ defined above,
or as a sum 
\begin{equation*}
	\gamma=\mu+\nu
\end{equation*}
of two elements $\mu\in(\mathfrak{t}\cap\mathfrak{m})_\mathbb{C}^*$ and $\nu\in\mathfrak{a}_\mathbb{C}^*$
where 
\begin{equation*}
	\mu=c_1\varepsilon_1+\ldots+c_{n-1}\varepsilon_{n-1}
\end{equation*}
and 
\begin{equation*}
	\nu\coloneqq c_n\varepsilon_n.
\end{equation*}
We will say that an element $\gamma\in (\mathfrak{t}_s)^*_\mathbb{C}$ is integer (respectively half-integer) if all $c_i$'s are integer (respectively half-integer).

The maximal compact subgroup $K$ of $G$ determined by the Cartan involution $\theta$
 is exactly the spin double cover of the maximal compact subgroup 
 \begin{equation*}
 	K_1\coloneqq \{\mathrm{diag}(g,1)\mid g\in \mathrm{SO}(m+1)\}\cong \mathrm{SO}(m+1)
 \end{equation*}
of $\mathrm{SO}(m+1,1)$
so that $K\cong\mathrm{Spin}(m+1)$.
Let $A$, $N$ and $\bar{N}$ be the connected analytic subgroups of $G$ with corresponding Lie algebras $\mathfrak{a}$, $\mathfrak{n}$ and $\mathfrak{\bar{n}}$ respectively. Then, the closed subgroup
\begin{equation*}
	M\coloneqq Z_{K}(A)
	\end{equation*}
of $G$ is the spin double cover of the compact subgroup 
\begin{equation*}
	M_1\coloneqq \{\mathrm{diag}(g,I_{2})\mid g\in \mathrm{SO}(m)\}\cong\mathrm{SO}(m)
\end{equation*}
of $\mathrm{SO}(m+1,1)$ so that $M\cong\mathrm{Spin}(m)$. Let $P\coloneqq MAN$ and $\bar{P}\coloneqq MA\bar{N}$ be the corresponding minimal parabolic subgroups of $G$. 

Under the identification $ \mathbb{R}^m\cong \mathfrak{n}$ given by 
\begin{equation*}
	\alpha\in\mathbb{R}^m\mapsto X_{\alpha},
\end{equation*} we obtain a natural identification $\mathfrak{n}^*\cong \mathbb{R}^m$.  Fix an element
$\xi_0\in\mathfrak{n}^*$ such that
\begin{equation*}
	\xi_0=(0,\ldots,0,1)
\end{equation*}
via this identification,
and set $M'\coloneqq \mathrm{Stab}_{MA}(\xi_0)$. We note that $M'$ is the spin double cover of the subgroup 
\begin{equation*}
	M_1'\coloneqq \{\mathrm{diag}(g,I_3)\mid g\in\mathrm{SO}(m-1))\}
\end{equation*}
of $\mathrm{SO}(m+1,1)$ so that $M'\cong \mathrm{Spin}(m-1)$ while the Lie algebra $\mathfrak{m}'$ of $M'$ is isomorphic to  $\mathfrak{so}(m-1,\mathbb{R})$.

 Consider a unitary representation $(\tau,V_\tau)$ of $M'$ and let
\begin{equation*}
	I_{P,\tau}\coloneqq \mathrm{Ind}_{M'\ltimes N}^P(\tau\otimes e^{i\xi_0})
\end{equation*}
be the unitarily induced representation consisting of measurable functions $h:M'N\rightarrow V_\tau$ such that 

	\begin{itemize}
	\item[(a)] $h(pm'n)=(\tau\otimes e^{i\xi_0})(m',n)^{-1}h(p) $ for all $p\in P$, $m'\in M'$ and $n\in N$,
	\item[(b)] $\langle h,h\rangle <\infty$
\end{itemize}
where, for $h_1,h_2\in I_{P,\tau}$,
\begin{equation*}
	\langle h_1,h_2\rangle\coloneqq \int_{MA/M'}\langle h_1(ma),h_2(ma)\rangle_\tau d_lma.
\end{equation*}
Here $\langle\cdot,\cdot\rangle_{\tau}$ stands for the $M'$-invariant form of $V_\tau$ and $d_lma$ is a left $(MA)$-invariant measure on the homogeneous space $MA/M'$. As usual, the left $P$-action on a function $h\in I_{P,\tau}$ is given by left-translation.

\begin{proposition}[\cite{mackeylittlegroup,wolfpar}]\label{Pdiscrete}
	Any infinite-dimensional irreducible unitary representation of $P$ is isomorphic to $I_{P,\tau}$ for a unique (up to isomorphism) irreducible finite-dimensional unitary representation $\tau$ of $M'$. Moreover, all the $I_{P,\tau}$ are discret series representations of P.
\end{proposition}

\begin{remark}\label{remarkdiscrete}
	The fact that all the $I_{P, \tau}$'s are discret series representations of P can also be seen from Duflo's orbit method. Actually all the $I_{P,\tau}$'s are associated to regular coadjoint orbits with compact stabilizers.
	\end{remark}

\section{Decomposition of the tensor product}\label{decomposition1}

 In this section, we provide the direct integral decomposition of the tensor product representation $\pi_1\widehat{\otimes}\pi_2$ in the case where $\pi_1$ is a unitary pricipal series representation and $\pi_2$ is any nontrivial unitary irreducible representation of $G=\mathrm{Spin}(m+1,1)$.

Consider such a unitary principal series representation 
\begin{equation*}
	\pi_1=\mathrm{Ind}_{P}^G\big(\sigma\otimes e^{\nu}\otimes 1_{N}\big),
\end{equation*}
where $\sigma\in \widehat{M}$, i.e., $\sigma$ is a finite dimensional irreducible representation of $M$, say of highest weight $\mu\in(\mathfrak{t}\cap\mathfrak{m})^*_\mathbb{C}$, and  $e^{\nu}$ is a unitary character of $A$, i.e., $\nu\in i\mathfrak{a}^*$ and $de^\nu=\nu$.

For the sake of completeness, let us recall that we define the induced representation 
\begin{equation*}
	\mathrm{Ind}_{P}^G(\sigma\otimes e^{\nu}\otimes 1)=:I(\mu,\nu)
\end{equation*}
as the space of measurable functions $h:G\rightarrow V_\sigma$ satisfying the invariant condition 
\begin{equation*}
	h(gman)=\sigma(m)^{-1}e^{(-\nu-\rho')(\log a)}h(g), \hspace{3mm}\forall g\in G, m\in M, n\in N,
	\end{equation*}
where, as above, $\rho'$ stands for the half-sum of positive restricted roots of $\Sigma(\mathfrak{g},\mathfrak{a})$, and 
\begin{equation*}
	h\vert_K\in L^2(K).
\end{equation*}

Let $\pi_2$ be any nontrivial unitary irreducible representation of $G$. Then $\pi_2$ can be expressed as the induced representation $\mathrm{Ind}_G^G(\pi_2)$ consisting of all functions $h:G\rightarrow V_{\pi_2}$ satisfying 
\begin{equation*}
	h(gx)=x^{-1}(h(g)), \text{ for all } g,x\in G.
\end{equation*}
One can easily see that $\pi_2\cong \mathrm{Ind}_G^G(\pi_2)$. Indeed, the evaluation map 
\begin{equation*}
	\mathrm{ev}_1:\mathrm{Ind}_G^{G}(\pi_2)\rightarrow \pi_2
\end{equation*}
where 
\begin{equation*}
	\mathrm{ev}_1(h)\coloneqq h(1), \quad h\in\mathrm{Ind}_G^G(\pi_2)
\end{equation*}
is a $G$-equivariant isomorphism.

Then
\begin{align*}
	\pi_1\widehat{\otimes}\pi_2&=\mathrm{Ind}^G_{P}(\sigma\otimes e^{\nu}\otimes 1) \widehat{\otimes}\pi_2 \\
	&=\mathrm{Ind}_{P}^{G}(\sigma\otimes e^{\nu}\otimes 1)\widehat{\otimes}\mathrm{Ind}_G^G(\pi_2)\\
\end{align*}

\begin{lemma}\label{lemmatensor}
	As unitary $G$-representations, we have 
	\begin{equation}\label{stop1}
		\mathrm{Ind}^G_P(\sigma\otimes e^{\nu}\otimes 1)\widehat{\otimes}\mathrm{Ind}_G^G(\pi_2)\cong \mathrm{Ind}_{P}^{G}\big((\sigma\otimes e^{\nu}\otimes 1)\widehat{\otimes}(\pi_2)_{\vert P}\big).
	\end{equation}
\end{lemma}

For the proof of Lemma \ref{lemmatensor}, we will use \cite[Theorem 7.2]{mackeyannals1} as described in what follows.
Let $G'$ be a locally compact group and $H_1$ and $H_2$ be closed subgroups of $G'$ such that they are discretely related, i.e., there exist countably many elements $x_1,x_2,\ldots$ such that 
\begin{equation*}
	\vert G'\setminus \bigcup\limits_{i\geq 1} H_1x_iH_2\vert =0,
\end{equation*}
where $\vert \cdot\vert$ stands for the Haar measure on $G'$. If $L_1$ and $L_2$ are representations of $H_1$ and $H_2$ respectively, then 
\begin{equation*}
	\mathrm{Ind}_{H_1}^{G'}\big( L_1\big)\widehat{\otimes}\mathrm{Ind}_{H_2}^{G'}\big(L_2\big)\cong \bigoplus\limits_{F(x,y)\in\mathcal{F}}\mathrm{Ind}^{G'}_{x^{-1}H_1x\cap y^{-1}H_2y}(N^{x,y})
\end{equation*}
where $\mathcal{F}$ is the set of double cosets $H_1xy^{-1}H_2$ with $\vert H_1xy^{-1}H_2\vert \neq 0$, $F(x,y)\coloneqq H_1xy^{-1}H_2$, and $N^{x,y}$ is the representation of $x^{-1}H_1x\cap y^{-1}H_2y$ given by
\begin{equation*}
	x^{-1}H_1x\cap y^{-1}H_2y\ni s\mapsto (L_1)_{xsx^{-1}}\otimes (L_2)_{ysy^{-1}}. 
\end{equation*}
We note that the above construction is well-defined in the sense that $N^{x,y}$ is determined up to unitary equivalence by $F(x,y)$.

\begin{proof}[Proof of Lemma \ref{lemmatensor}]
	In the above Mackey's result, set $G'\coloneqq G=\mathrm{Spin}(m+1,1)$, $H_1\coloneqq P$, $H_2\coloneqq G$, $L_1\coloneqq\sigma\otimes e^{\nu}\otimes 1$ and $L_2\coloneqq \pi_2$. Since $Pxy^{-1}G=G$  for every $x,y\in G$, there is a unique coset, so that $\mathcal{F}=\{G\}$ and 
	\begin{equation*}
			\mathrm{Ind}_{P}^{G}( \sigma\otimes e^{\nu}\otimes 1)\widehat{\otimes}\mathrm{Ind}_{G}^{G}(\pi_2)\cong \mathrm{Ind}^{G}_{x^{-1}Px\cap y^{-1}Gy}(N^{x,y})
	\end{equation*}
for any $(x,y)\in G\times G$. Choose $(x,y)=(1,1)$ so that $x^{-1}Px\cap y^{-1}Gy=P$ and the representation $N^{x,y}=N^{1,1}$ of $P$ is
\begin{equation*}
	N^{1,1}=(\sigma \otimes e^{\nu}\otimes 1)\widehat{\otimes}(\pi_2)_{\vert P}.
	\end{equation*}
Hence 
\begin{equation*}
	\mathrm{Ind}_{P}^{G}( \sigma\otimes e^{\nu}\otimes 1)\widehat{\otimes}\mathrm{Ind}_{G}^{G}(\pi_2)\cong \mathrm{Ind}^{G}_{P}\big((\sigma \otimes e^{\nu}\otimes 1)\widehat{\otimes}(\pi_2)_{\vert P}\big).
	\end{equation*}
	\end{proof}

In \cite[Section 3]{gangoshima}, an explicit branching law for $\pi_2\vert_{P}$ is given for all $\pi_2\in \widehat{G}$. More precisely, when $\pi_2$ is  nontrivial, $\pi_2\vert_P$ decomposes into a finite direct sum of  discrete series representations of $P$
\begin{equation}\label{Pdecomposition}
	\pi_2\vert_{ P}\cong \bigoplus\limits_{j=1}^kT_j,
\end{equation}
where 
\begin{equation*}
	T_j\coloneqq \mathrm{Ind}_{M'N}^P(\tau_j\otimes e^{i\xi_0})
\end{equation*}
for some finite-dimensional irreducible representation $\tau_j$ of $M'$.
Recall that $\xi_0\coloneqq (0,\ldots,0,1)\in \mathfrak{n}^*$ and $M'\coloneqq \mathrm{Stab}_{MA}(\xi_0)\cong \mathrm{Spin}(m-1)$. We will give the explicit branching law of $\pi_2\vert_P$ for the cases we need later on. Let us note that recently Arends, Bang-Jensen and Frahm have also obtained the branching law of $\pi_2\vert_P$ by different methods \cite{arends}.

Let us now replace \eqref{Pdecomposition} in \eqref{stop1} to obtain
\begin{align}\label{stop3}
	\begin{split}
	\pi_1\widehat{\otimes}\pi_2&\cong \bigoplus\limits_{j=1}^k\mathrm{Ind}_P^G\big((\sigma\otimes e^{\nu}\otimes 1)\widehat{\otimes} T_j\big)\\
	&\cong\bigoplus\limits_{j=1}^k\mathrm{Ind}_P^G\big(\mathrm{Ind}_{P}^P(\sigma\otimes e^{\nu}\otimes 1)\widehat{\otimes}\mathrm{Ind}_{M'N}^{P}(\tau_j\otimes e^{i\xi_0})\big).
	\end{split}
\end{align}
By using again \cite[Theorem 7.2]{mackeyannals1}, we obtain that

\begin{align}\label{stop2}
	\begin{split}
&\mathrm{Ind}_{P}^P(\sigma\otimes e^{\nu}\otimes 1)
\widehat{\otimes}\mathrm{Ind}_{M'N}^{P}(\tau_j\otimes e^{i\xi_0})\\
&
\cong \mathrm{Ind}_{M'N}^P\big((\sigma\otimes e^{\nu}\otimes 1)_{\vert M'N}\otimes (\tau_j\otimes e^{i\xi_0})\big)\\
&\cong \mathrm{Ind}_{M'N}^P\big(\sigma_{\vert M'}\otimes (\tau_j\otimes e^{i\xi_0})\big)\\
&\cong \mathrm{Ind}_{M'N}^P\big((\sigma_{\vert M'}\otimes \tau_j)\otimes e^{i\xi_0})\big)
\end{split}
\end{align}

Let $(\sigma_1,\ldots,\sigma_{[\frac{m}{2}]})$ be the highest weight of the irreducible 
$M$-representation $\sigma$. According to \cite[Theorem 8.1.3, 8.1.4]{goodman},
\begin{equation}\label{finiterestr}
	\sigma_{\vert M'}\cong\bigoplus\limits_{\beta\in \mathcal{B}}F_{\beta}
\end{equation}
where 
$F_\beta$ is the irreducible finite-dimensional representation of highest weight $\beta\in\mathcal{B}$. 

If $m$ is odd,
	$\mathcal{B}$ is the (finite) set of all $\beta=(\beta_1,\ldots,\beta_{\frac{m-1}{2}})\in(\mathfrak{t}\cap\mathfrak{m}')^*$ such that
	\begin{itemize}
	 \item[(1)] $\beta$ is integer if $\sigma$ is integer or $\beta$ is half-integer if $\sigma$ is half-integer, and 
	 \item[(2)] 
	$	\sigma_1\geq \beta_1\geq\sigma_2\geq \ldots \geq \sigma_{\frac{m-3}{2}}\geq \beta_{\frac{m-3}{2}}\geq \sigma_{\frac{m-1}{2}}\geq \vert\beta_{\frac{m-1}{2}}\vert.$
\end{itemize}

If $m$ is even, $\mathcal{B}$ is the (finite) set of all
$\beta=(\beta_1,\ldots,\beta_{\frac{m-2}{2}})\in(\mathfrak{t}\cap\mathfrak{m}')^*$ such that
	\begin{itemize}
	\item[(1)] $\beta$ is integer if $\sigma$ is integer or $\beta$ is half-integer if $\sigma$ is half-integer, and 
 \item[(2)]
$
	\sigma_1\geq \beta_1\geq\sigma_2\geq \ldots \geq \sigma_{\frac{m-2}{2}}\geq \beta_{\frac{m-2}{2}}\geq \vert\sigma_{\frac{m}{2}}\vert.
$
\end{itemize}

We replace \eqref{finiterestr} in \eqref{stop2} to obtain
\begin{equation*}
	\mathrm{Ind}_{P}^P(\sigma\otimes e^{\nu}\otimes 1)\widehat{\otimes}\mathrm{Ind}_{M'N}^{P}(\tau_j\otimes e^{i\xi_0})\cong \bigoplus\limits_{\beta\in \mathcal{B}}\mathrm{Ind}_{M'N}^P\big((F_\beta\otimes \tau_j)\otimes e^{i\xi_0}\big)
\end{equation*}
and then from \eqref{stop3} we obtain
\begin{equation}\label{stop4}
	\pi_1\widehat{\otimes}\pi_2 \cong\bigoplus\limits_{j=1}^k\bigoplus\limits_{\beta\in\mathcal{B}}\mathrm{Ind}_P^G\big(\mathrm{Ind}_{M'N}^P\big((F_\beta\otimes\tau_j)\otimes e^{i\xi_0}\big)\big).
\end{equation}

Let us now have a closer look at the (finite-dimensional) tensor product representation $F_\beta\otimes\tau_j$ of $M'$. Assume that the irreducible representation $\tau_j$ has highest weight
\begin{equation*}
	\gamma_j=(\gamma_j^1,\ldots ,\gamma_j^{m'})\in(\mathfrak{t}\cap\mathfrak{m}')_\mathbb{C}^*,
\end{equation*}
where $m'\coloneqq [\frac{m-1}{2}]$. According to \cite[Corollary 3.6]{kumar}, 
\begin{equation}\label{partition}
	F_\beta\otimes\tau_j=\bigoplus\limits_{\delta\in(\mathfrak{t}\cap\mathfrak{m}')^*_{\mathbb{C},+}}n^j_\beta(\delta)F_{\delta}
\end{equation}
where 
 $(\mathfrak{t}\cap\mathfrak{m}')^*_{\mathbb{C},+}$ stands for the set of integral dominant elements $\delta$ of $(\mathfrak{t}\cap\mathfrak{m}')^*_\mathbb{C}$, $F_\delta$ for the irreducible finite-dimensional representation of highest weight $\delta\in(\mathfrak{t}\cap\mathfrak{m}')^*_{\mathbb{C},+}$, and
\begin{equation*}
	n^j_\beta(\delta)\coloneqq \sum\limits_{v,w\in W}\varepsilon(v)\varepsilon(w)\mathcal{P}(v(\gamma_j+\rho)-w(\delta+\rho)+\beta).
\end{equation*}
Here $\varepsilon(v)=(-1)^{\vert v\vert}$, where $\vert v\vert $ stands for the length of the Weyl group element $v$, and $\mathcal{P}(\cdot)$ stands for the Kostant partition function \cite[p. 315]{Knapp1}. Remark that $n_\beta^j(\delta)\in \mathbb{N}\cup\{0\}$ and it is nonzero only for a finite number of $\delta$'s.

Let us replace \eqref{partition} in \eqref{stop4} to obtain
\begin{equation}\label{thistensor}
	\pi_1\widehat{\otimes}\pi_2 \cong\bigoplus\limits_{j=1}^k\bigoplus\limits_{\beta\in\mathcal{B}}\bigoplus\limits_{\delta\in(\mathfrak{t}\cap\mathfrak{m}')^*_{\mathbb{C},+}}n_\beta ^j(\delta)\mathrm{Ind}_P^G\big(\mathrm{Ind}_{M'N}^P\big(F_\delta\otimes e^{i\xi_0}\big)\big).
\end{equation}
To simplify, set 
\begin{equation*}
	T_\delta\coloneqq  \mathrm{Ind}_{M'N}^P(F_\delta\otimes e^{i\xi_0})
\end{equation*}
for every $\delta\in (\mathfrak{t}\cap\mathfrak{m}')^*_{\mathbb{C},+}$. 
Then 
\begin{equation}\label{thistensor2}
	\pi_1\widehat{\otimes}\pi_2\cong \bigoplus\limits_{j=1}^k\bigoplus\limits_{\beta\in\mathcal{B}}\bigoplus\limits_{\delta\in(\mathfrak{t}\cap\mathfrak{m}')^*_{\mathbb{C},+}}n_\beta ^j(\delta)\mathrm{Ind}_P^G\big(T_\delta\big).
\end{equation}
It is clear that in order to decompose $\pi_1\widehat{\otimes}\pi_2$, it suffices to decompose $\mathrm{Ind}_P^G\big(T_\delta\big)$ for every $\delta\in (\mathfrak{t}\cap\mathfrak{m}')^*_{\mathbb{C},+}$. As we pointed out in Proposition \ref{Pdiscrete}, 
all $T_\delta$'s are discrete series representations of $P$. This fact will enable us to apply Anh-Mackey's Frobenious reciprocity which allows us to provide an explicit decomposition of $\mathrm{Ind}_P^G(T_\delta)$. 

In what follows, we will give the details about the decomposition of $\mathrm{Ind}_P^G(T_\delta)$.
%
Fix a representation $T_\delta$, $\delta\in(\mathfrak{t}\cap\mathfrak{m}')^*_{\mathbb{C},+}$, and let 
\begin{equation*}
\mathcal{A}_\delta\coloneqq \{\pi\in \widehat{G}_{\mathrm{temp}}\mid \mathrm{Hom}_P(T_\delta,\pi)\neq \{0\}\},
\end{equation*}
where $\widehat{G}_{\mathrm{temp}}$ stands for the tempered dual of $G$.

\begin{proposition}\label{lastdecomp}
	For every $\delta\in(\mathfrak{t}\cap\mathfrak{m}')^*_{\mathbb{C},+}$, 
	\begin{equation*}
		\mathrm{Ind}_P^G(T_\delta)\cong \int\limits_{\widehat{G}}^\oplus\chi_{\mathcal{A}_\delta}(\pi)\mathcal{H}_\pi dm_G(\pi),
	\end{equation*}
	where $m_G$ stands for the Plancherel measure of $G$.
\end{proposition}

For the proof of Proposition \ref{lastdecomp}, we will use \cite[Corollary 1.10]{anh}, where the author establishes a Frobenius reciprocity as follows. Assume that  $G'$ is a locally compact separable topological group and $H'$ is a closed subgroup of $G'$ such that they are both of Type I (see \cite[p. 129]{kirillovnew}) and let $m_{G'}$ and $m_{H'}$ be the Plancherel measures of $G'$ and $H'$ respectively. If $\omega(x,y)$ and $n(x,y)$, $(x,y)\in\widehat{H'}\times\widehat{G'}$, are $m_{H'}\times m_{G'}$-measurable functions and $n(x,y)$ is a countable cardinal, then for $m_{H'}$-almost all $x$ 
\begin{equation*}
	\mathrm{Ind}_{H'}^{G'}\mathcal{V}^x\cong \int\limits_{\widehat{G'}}^\oplus (\mathcal{H}^y)^{\oplus n(x,y)}\omega(x,y)dm_{G'}(y)
\end{equation*}
if and only if for $m_{G'}$-almost all $y$
\begin{equation*}
	\mathcal{H}^y\vert_{H'}\cong\int\limits_{\widehat{H'}}^\oplus (\mathcal{V}^x)^{\oplus n(x,y)}\omega(x,y)dm_{H'}(x),
\end{equation*}
where $\mathcal{V}^x$ (respectively $\mathcal{H}^y$) stands for the Hilbert space of the unitary irreducible representation $x=(\phi_x,\mathcal{V}^x)$ of $ H'$ (respectively $y=(\pi_y,\mathcal{H}^y)$ of $G'$).

\begin{proof}[Proof of Proposition \ref{lastdecomp}] We will apply \cite[Corollary 1.10]{anh} for $G'\coloneqq G$ and $H'\coloneqq P$. We start by noting that $G$ and $P$ are almost-algebraic groups so that they are both of type I \cite{dixmiertype1}. Moreover, according to \cite{gangoshima}, for every $y=(\pi_y,\mathcal{V}^y)\in \widehat{G}$,
	\begin{equation*}
		\mathcal{V}^y\vert_P=\bigoplus\limits_{\delta \text{ s.t. } y\in\mathcal{A}_\delta}T_\delta
	\end{equation*}
	so that 
	\begin{equation*}
		n(x_\delta,y)=\omega(x_\delta,y)=
		\begin{cases}
			1 &\text{ if } \pi_y\in\mathcal{A}_\delta\\
			0 & \text{ otherwise}
		\end{cases}=\chi_{\mathcal{A}_\delta}(y)
	\end{equation*}
	where $x_\delta\in\widehat{P}$ corresponds to $T_\delta\coloneqq \mathrm{Ind}_{M'N}^P(\delta)$.
	Hence
	\begin{equation*}
		\mathrm{Ind}_P^G(T_\delta)\cong\int\limits_{\widehat{G}}^\oplus \omega(x_\delta,y) (\mathcal{H}^y)^{\oplus n(x_\delta,y)}dm_G(y)\cong \int\limits_{\widehat{G}}^\oplus \chi_{\mathcal{A}_\delta}(y)\mathcal{H}^ydm_G(y).
	\end{equation*}
\end{proof}

Let us consider the Plancherel decomposition \cite{knappstein1}
\begin{equation}\label{Plancherel}
	L^2(G)\cong\widehat{\bigoplus\limits_{y\in\widehat{G}_{\mathrm{ds}}}}\mathcal{H}^y\widehat{\otimes}(\mathcal{H}^y)' \oplus\widehat{\bigoplus\limits_{\phi\in \widehat{M}}}\int_{(0,\infty)}^\oplus  \mathcal{H}_{\phi,t\lambda_0}\widehat{\otimes}\mathcal{H}_{\phi,t\lambda_0}^*d\nu_\phi(t)
\end{equation}
where $\widehat{G}_{\mathrm{ds}}$ stands for the subset of $\widehat{G}_{\mathrm{temp}}$ consisting of the discrete series representations of $G$, $\mathcal{H}_{\phi,t\lambda_0}$ is the unitary principal series representation 
\begin{equation*}
\mathrm{Ind}_{P}^G\big(\phi\otimes e^{it\lambda_0}\otimes 1_{N}\big),
\end{equation*}
and $d\nu_\phi(t)$
 is explicitly described in \cite{rankonePlancherel} as follows.
 
 If $G=\mathrm{Spin}(2n,1)$ then 
\begin{equation*}
d\nu_\phi(t)\coloneqq d\nu_\phi^1(t)\coloneqq \mathrm{tanh}(t\pi)f_\phi(t)dt
\end{equation*}
if $\phi$ is integer, and 
\begin{equation*}
	d\nu_\phi(t)\coloneqq d\nu_\phi^1(t)\coloneqq \mathrm{coth}(t\pi)f_\phi(t)dt
\end{equation*}
if $\phi$ is half-integer, where $dt$ stands for the Lebesgue measure of $\mathbb{R}$. In both cases
\begin{equation*}
	f_\phi(t)\coloneqq \frac{t\mathrm{dim}(\phi)}{\pi^{n}\Gamma(n)}{\displaystyle\prod\limits_{i=1}^{n-1}(t^2+(a_i+n-i-\frac{1}{2})^2)}
\end{equation*}
with $(a_1,\ldots ,a_{n-1})$ being the highest weight of $\phi$. 

If $G=\mathrm{Spin}(2n-1,1)$ then there is no discrete series representations while
\begin{equation*}
	d\nu_\phi(t)\coloneqq \frac{\mathrm{dim}(\phi)}{\pi^{n-\frac{1}{2}}\Gamma(n-\frac{1}{2})}{\displaystyle\prod\limits_{i=1}^{n-1}(t^2+(a_i+n-i-1)^2)}dt
\end{equation*}
with $(a_1,\ldots ,a_{n-1})$ being the highest weight of $\phi$. As before $dt$ stands for the Lebesgue measure of $\mathbb{R}$. 

According to the Plancherel formula \eqref{Plancherel}, the tempered dual of $G$ consists essentially of discrete series and unitary principal series representations in the sense that the complement of these representations in $\widehat{G}$ is negligible with respect to the Plancherel measure $m_G$ of $G$. As a consequence, in order to determine $\mathcal{A}_\delta$ and then use Proposition \ref{lastdecomp} to compute explicitly the decomposition of $\mathrm{Ind}_P^G(T_{\delta})$, it suffices to determine the discrete series and the unitary principal series representations of $G$ belonging to $\mathcal{A}_{\delta}$. This comes down to the decomposition $\pi\vert_P$ for every discrete and unitary principal series representation $\pi$ of $G$. 
Let us describe these decompositions as given in \cite{gangoshima}.

Assume that $G=\mathrm{Spin}(2n,1)$. 
Then $M'=\mathrm{Spin}(2(n-1))$ and fix an integral dominant $\delta=(b_1,\ldots,b_{n-1})\in (\mathfrak{t}\cap \mathfrak{m}')^*_{\mathbb{C},+}$, i.e.,
\begin{equation*}
	b_1\geq \ldots \geq b_{n-2}\geq  \vert b_{n-1}\vert
\end{equation*}
with all $b_i$ integers or all half-integers. 

\vspace{2mm}
\textbf{Case I (discrete series representation):}

Let
\begin{equation*}
	\gamma=	(a_1+n-\frac{1}{2},\ldots, a_{n-1}+\frac{3}{2},a_n+\frac{1}{2})
\end{equation*}
with $a_1\geq \ldots\geq a_{n-1}\geq a_n\geq 0$ all integers or all half-integers.
Define $\pi_0^\pm(\gamma)$ to be the discrete series representation with lowest $K$-type $V_{K,\lambda^+}$, where 
\begin{equation*}
	\lambda^\pm\coloneqq (a_1+1,\ldots,a_{n-1}+1,\pm(a_n+\frac{1}{2}))\in\mathfrak{t}^*_\mathbb{C}.
\end{equation*}
Here $V_{K,\lambda^\pm}$ is the finite-dimensional irreducible representation of $K$ with highest weight $\lambda^\pm$. Note that $\gamma$ is the infinitesimal character of $\pi_0^\pm(\gamma)$.
According to \cite[Theorem 3.22]{gangoshima},
\begin{equation*}
	\pi_0^{\pm}(\gamma)\vert_{P}\cong \bigoplus\limits_{\tau}\mathrm{Ind}_{M'N}^P(V_{M',\tau}\otimes e^{i\xi_0})
\end{equation*}
where $\tau=(c_1,\ldots,c_{n-1})$ runs over tuples such that
\begin{equation*}\label{decompositionLambda0}
	a_1+1\geq c_1\geq a_{2}+1\geq\ldots \geq c_{n-2}\geq a_{n-1}+1\geq  \mp c_{n-1} \geq a_n+1
\end{equation*}
and $c_i-a_1\in\mathbb{Z}$. Hence, the representations $\pi_0^{\pm}(\gamma)$ contribute to $\mathcal{A}_\delta$ if and only if 
\begin{equation}\label{cnd1}
	\begin{array}{lc}
		&a_1+1\geq b_1\geq a_{2}+1\geq\ldots \geq b_{n-2}\geq a_{n-1}+1\geq  \mp b_{n-1} \geq a_n+1\\
		\text{ and }&
		b_i-a_1\in\mathbb{Z}. 
	\end{array}
\end{equation}
Denote by $\mathcal{D}_0(\delta)$ the set of all $\gamma$'s satisfying the above condition \eqref{cnd1}.

\vspace{2mm}
\textbf{Case II (unitary principal series representations):} \\
Let
\begin{equation*}
	\gamma=(a_1+n-\frac{3}{2},\ldots, a_{n-1}+\frac{1}{2},a_n)
\end{equation*}
where $a_1\geq \ldots\geq a_{n-1}\geq0$ and $a_n\in i\mathbb{R}$ and set 
\begin{equation*}
	\mu=(a_1,\ldots,a_{n-1})
\end{equation*}
and 
\begin{equation*}
	\nu=a_n\lambda_0.
\end{equation*}
Then
\begin{equation*}
	\pi(\gamma)\coloneqq \mathrm{Ind}_{MA\bar{N}}^G(V_{M,\mu}\otimes e^{\nu-\rho'}\otimes1_{\bar{N}})
\end{equation*}
is a unitary principal series representation with infinitesimal character $\gamma$.
According to \cite[Theorem 3.20]{gangoshima},
\begin{equation*}
	\pi(\gamma)\vert_{ P}\cong \bigoplus\limits_{\tau}\mathrm{Ind}_{M'N}^P(V_{M',\tau}\otimes e^{i\xi_0}),
\end{equation*}
where $\tau=(c_1,\ldots,c_{n-1})$ runs over tuples such that 
\begin{equation*}
	a_1\geq c_1\geq a_2\geq c_2\geq \ldots\geq a_{n-1}\geq \vert c_{n-1}\vert
\end{equation*}
and $c_i-a_1\in\mathbb{Z}$.
Hence $\pi(\gamma)$ contributes to $\mathcal{A}_\delta$ if and only if \begin{equation}\label{cnd2}
	\begin{array}{lc}
		&a_1\geq b_1\geq a_2\geq b_2\geq \ldots\geq a_{n-1}\geq \vert b_{n-1}\vert\\
		\text{ and } &b_i-a_1\in\mathbb{Z}
	\end{array}
\end{equation}
Denote by $\mathcal{C}_0(\delta)$ the set of all $(n-1)$-tuples $(a_1,\ldots, a_{n-1})$ satisfying the above condition \eqref{cnd2}.

\begin{proposition}\label{thisprop}
	Let $G=\mathrm{Spin}(2n,1)$.
	\begin{itemize}
		\item[(1)] If $\delta \in(\mathfrak{t}\cap\mathfrak{m}')^*_{\mathbb{C},+}$ is integer, then 
		\begin{equation*}
			\mathrm{Ind}_P^G(T_\delta)\cong
			\widehat{\bigoplus\limits_{\gamma\in \mathcal{D}_0(\delta)}}\mathcal{H}_{\pi_0^+(\gamma)}\oplus\widehat{\bigoplus\limits_{\gamma\in \mathcal{D}_0(\delta)}}\mathcal{H}_{\pi_0^-(\gamma)} \oplus
			\widehat{\bigoplus\limits_{\phi\in \mathcal{C}_0(\delta)}}\int_{(0,\infty)}^\oplus  \mathcal{H}_{\phi,t\lambda_0}d\nu_{\phi}^1(t),
		\end{equation*}
		where
		\begin{equation*}
			d\nu_\phi^1(t)\coloneqq \mathrm{tanh}(t\pi)\frac{t\mathrm{dim}(\phi)}{\pi^{n}\Gamma(n)}{\displaystyle\prod\limits_{i=1}^{n-1}(t^2+(a_i+n-i-\frac{1}{2})^2)}dt.
		\end{equation*}
		\vspace{2mm}
		\item[(2)] If $\delta \in(\mathfrak{t}\cap\mathfrak{m}')^*_{\mathbb{C},+}$ is half-integer, then
		\begin{equation*}
			\mathrm{Ind}_P^G(T_\delta)\cong
			\widehat{\bigoplus\limits_{\gamma\in \mathcal{D}_0(\delta)}}\mathcal{H}_{\pi_0^+(\gamma)}\oplus\widehat{\bigoplus\limits_{\gamma\in \mathcal{D}_0(\delta)}}\mathcal{H}_{\pi_0^-(\gamma)} \oplus
			\widehat{\bigoplus\limits_{\phi\in \mathcal{C}_0(\delta)}}\int_{(0,\infty)}^\oplus  \mathcal{H}_{\phi,t\lambda_0}d\nu_{\phi}^2(t),
		\end{equation*}
		where
		\begin{equation*}
			d\nu_\phi^2(t)\coloneqq \mathrm{coth}(t\pi)\frac{t\mathrm{dim}(\phi)}{\pi^{n}\Gamma(n)}{\displaystyle\prod\limits_{i=1}^{n-1}(t^2+(a_i+n-i-\frac{1}{2})^2)}dt.
		\end{equation*}
	\end{itemize}
\end{proposition}

\begin{proof}
It is a direct consequence of Proposition \ref{lastdecomp}, the explicit Plancherel formula \eqref{Plancherel} for $\mathrm{Spin}(2n,1)$ and the above branching laws.
\end{proof}

Assume now that $G=\mathrm{Spin}(2n-1,1)$. This
case is similar and simpler than the previous one. We have $M'\coloneqq \mathrm{Spin}(2(n-2)+1)$ and fix an integral dominant element $\delta=(b_1,\ldots,b_{n-2})\in (\mathfrak{t}\cap \mathfrak{m}')^*_{\mathbb{C},+}$, i.e.,
\begin{equation*}
	b_1\geq \ldots \geq b_{n-3}\geq   b_{n-2}\geq 0
\end{equation*}
with all $b_i$'s being integers or all half-integers. 
Let 
\begin{equation*}
	\gamma=(a_1+n-2,a_{2}+n-3,\ldots, a_{n-1},a_n)
\end{equation*}
with $a_1\geq a_2\geq\ldots\geq a_{n-1}\geq 0$, $a_1,\ldots,a_{n-1}$ being all integers or all half-integers, and $a_n\in i\mathbb{R}$. Set
\begin{equation*}
\mu\coloneqq (a_1,\ldots,a_{n-1})
\end{equation*}
and 
\begin{equation*}
\nu\coloneqq a_n\lambda_0.
\end{equation*}
The representation
\begin{equation*}
	\pi(\gamma)\coloneqq \mathrm{Ind}_{MA\bar{N}}^G(V_{M,\mu}\otimes e^{\nu-\rho'}\otimes1_{\bar{N}}) 
\end{equation*}
is a unitary principal series representation of $G$.
According to \cite[Theorem 3.23]{gangoshima},
\begin{equation*}
	\pi(\gamma)\vert_{ P}\cong \bigoplus\limits_{\tau}\mathrm{Ind}_{M'N}^P(V_{M',\tau}\otimes e^{i\xi_0})
\end{equation*}
where $\tau=(c_1,\ldots,c_{n-2})$ runs over tuples such that
\begin{equation*}
	a_1\geq c_1\geq a_2\geq c_2\geq \ldots\geq a_{n-2}\geq c_{n-2}\geq \vert a_{n-1}\vert
\end{equation*}
and $c_i-a_1\in\mathbb{Z}$. Hence $\pi(\gamma)$ contributes to $\mathcal{A}_\delta$ if and only if 
\begin{equation}\label{cnd3}
	\begin{array}{lc}
		&a_1\geq b_1\geq a_2\geq b_2\geq \ldots\geq a_{n-2}\geq b_{n-2}\geq \vert a_{n-1}\vert\\
		\text{ and } & b_i-a_1\in\mathbb{Z}.
	\end{array}
\end{equation}
Denote by $\mathcal{C}_1(\delta)$ the set of all $(n-1)$-tuples $(a_1,\ldots ,a_{n-1})$ satisfying the above condition \eqref{cnd3}.


As a direct consequence of Proposition \ref{lastdecomp}, the explicit Plancherel formula \eqref{Plancherel} for $\mathrm{Spin}(2n-1,1)$ and the above branching laws, we obtain the following proposition.

\begin{proposition}\label{thisprop'}
Let $G=\mathrm{Spin}(2n-1,1)$. Then 
\begin{equation*}
	\mathrm{Ind}_P^G(T_\delta)\cong
	\widehat{\bigoplus\limits_{\phi\in \mathcal{C}_1(\delta)}}\int_{(0,\infty)}^\oplus  \mathcal{H}_{\phi,t\lambda_0}d\nu_\phi(t),
\end{equation*}
where 
\begin{equation*}
	d\nu_\phi(t)\coloneqq \frac{\mathrm{dim}(\phi)}{\pi^{n-\frac{1}{2}}\Gamma(n-\frac{1}{2})}{\displaystyle\prod\limits_{i=1}^{n-1}(t^2+(a_i+n-i-1)^2)}dt.
\end{equation*}
\end{proposition}

Now we are ready to present the main result of the article.

\begin{theorem}\label{mainresult} Let $G\coloneqq \mathrm{Spin}(m+1,1)$. Let $\pi_1$ be a unitary principal series representation
	\begin{equation*}
		\pi_1=\mathrm{Ind}_P^G(\sigma\otimes e^{\nu}\otimes 1)
	\end{equation*}
where $(\sigma_1,\ldots,\sigma_{[\frac{m}{2}]})$ is the highest weight of the irreducible 
$M$-representation $\sigma$. If $m$ is odd, let
 	$\mathcal{B}$ be the set of all $\beta=(\beta_1,\ldots,\beta_{\frac{m-1}{2}})\in(\mathfrak{t}\cap\mathfrak{m}')^*$ such that
\begin{itemize}
	\item[(1)] $\beta$ is integer if $\sigma$ is integer or $\beta$ is half-integer if $\sigma$ is half-integer, and 
	\item[(2)] 
	$	\sigma_1\geq \beta_1\geq\sigma_2\geq \ldots \geq \sigma_{\frac{m-3}{2}}\geq \beta_{\frac{m-3}{2}}\geq \sigma_{\frac{m-1}{2}}\geq \vert\beta_{\frac{m-1}{2}}\vert$.
\end{itemize}
If $m$ is even, let $\mathcal{B}$ be the set of all 
$\beta=(\beta_1,\ldots,\beta_{\frac{m-2}{2}})\in(\mathfrak{t}\cap\mathfrak{m}')^*$ such that
\begin{itemize}
	\item[(1)] $\beta$ is integer if $\sigma$ is integer or $\beta$ is half-integer if $\sigma$ is half-integer, and 
	\item[(2)]
	$
	\sigma_1\geq \beta_1\geq\sigma_2\geq \ldots \geq \sigma_{\frac{m-2}{2}}\geq \beta_{\frac{m-2}{2}}\geq \vert\sigma_{\frac{m}{2}}\vert.
	$
\end{itemize}
Let $\pi_2$ be any nontrivial unitary irreducible representation of $G$. Decompose $\pi_2$ into
	\begin{equation*} \pi_2\vert_{P}\cong\bigoplus\limits_{j=1}^k\mathrm{Ind}_{M'N}^P(\tau_j\otimes \xi_0)
	\end{equation*}
according to \cite{gangoshima} and let $\gamma_j\in (\mathfrak{t}\cap\mathfrak{m}')^*_{\mathbb{C},+}$ stand for the highest weight of $\tau_j$. For every $\delta \in(\mathfrak{t}\cap\mathfrak{m}')^*_{\mathbb{C},+}$, set 
\begin{equation*}
	n^j_\beta(\delta)\coloneqq \sum\limits_{v,w\in W}\varepsilon(v)\varepsilon(w)\mathcal{P}(v(\gamma_j+\rho)-w(\delta+\rho)+\beta),
\end{equation*}
where $\varepsilon(v)=(-1)^{\vert v\vert}$ and $\mathcal{P}(\cdot)$ stands for the Kostant partition function.
Then	
\begin{equation*}
	\pi_1\widehat{\otimes}\pi_2 \cong\bigoplus\limits_{j=1}^k\bigoplus\limits_{\beta\in\mathcal{B}}\bigoplus\limits_{\delta\in(\mathfrak{t}\cap\mathfrak{m}')^*_{\mathbb{C},+}}
	n_\beta ^j(\delta)\mathrm{Ind}_P^G(T_\delta)
	\end{equation*}
	where for every $\delta\in (\mathfrak{t}\cap\mathfrak{m}')^*_{\mathbb{C},+}$,
		\begin{equation*}
			T_\delta\coloneqq  \mathrm{Ind}_{M'N}^P(F_\delta\otimes e^{i\xi_0}),
	\end{equation*}
with $F_\delta$ being the finite-dimensional irreducible representation of $M'$ of highest weight $\delta$.
	
	\textbf{(A)} If $m$ is odd, set $n\coloneqq \frac{m+1}{2}$. Let $\mathcal{D}_0(\delta)$ be the set of $n$-tuples $(a_1,\ldots,a_n)$ and $\mathcal{C}_0(\delta)$ the set of $(n-1)$-tuples $(a_1',\ldots a_{n-1}')$ satisfying condition \eqref{cnd1} and \eqref{cnd2} respectively. 
	 \begin{itemize}
	 	\item[(1)] If $\delta$ is integer, then
	 	\begin{equation*}
	 		\mathrm{Ind}_P^G(T_\delta)\cong
	 	 \widehat{\bigoplus\limits_{\gamma\in \mathcal{D}_0(\delta)}}\mathcal{H}_{\pi_0^+(\gamma)}\oplus\widehat{\bigoplus\limits_{\gamma\in \mathcal{D}_0(\delta)}}\mathcal{H}_{\pi_0^-(\gamma)} \oplus\\
	 	\widehat{\bigoplus\limits_{\phi\in \mathcal{C}_0(\delta)}}\int_{(0,\infty)}^\oplus \mathcal{H}_{\phi,t\lambda_0}d\nu_{\phi}^1(t)
	 	\end{equation*}
 with
\begin{equation*}
d\nu_\phi^1(t)\coloneqq \mathrm{tanh}(t\pi)\frac{t\mathrm{dim}(\phi)}{\pi^{n}\Gamma(n)}{\displaystyle\prod\limits_{i=1}^{n-1}(t^2+(a_i'+n-i+\frac{1}{2})^2)}dt.
\end{equation*}

\item[(2)] If $\delta$ is half-integer, then
\begin{equation*}
	\mathrm{Ind}_P^G(T_\delta)\cong
	\widehat{\bigoplus\limits_{\gamma\in \mathcal{D}_0(\delta)}}\mathcal{H}_{\pi_0^+(\gamma)}\oplus\widehat{\bigoplus\limits_{\gamma\in \mathcal{D}_0(\delta)}}\mathcal{H}_{\pi_0^-(\gamma)} \oplus\\
	\widehat{\bigoplus\limits_{\phi\in \mathcal{C}_0(\delta)}}\int_{(0,\infty)}^\oplus \mathcal{H}_{\phi,t\lambda_0}d\nu_{\phi}^2(t)
\end{equation*}
with
\begin{equation*}
d\nu_\phi^2(t)\coloneqq \mathrm{coth}(t\pi)\frac{t\mathrm{dim}(\phi)}{\pi^{n}\Gamma(n)}{\displaystyle\prod\limits_{i=1}^{n-1}(t^2+(a_i'+n-i+\frac{1}{2})^2)}dt.
\end{equation*}
\end{itemize}

\textbf{(B)} If $m$ is even, set $n\coloneqq \frac{m+2}{2}$. Let $\mathcal{C}_1(\delta)$ be the set of $(n-1)$-tuples $(a_1',\ldots, a_{n-1}')$ satisfying condition \eqref{cnd3} and having the same parity with $\delta$. Then
		\begin{equation*}
		\mathrm{Ind}_P^G(T_\delta)\cong\bigg\{\widehat{\bigoplus\limits_{\phi\in \mathcal{C}_1(\delta)}}\int_{(0,\infty)}^\oplus  \mathcal{H}_{\phi,t\lambda_0}d\nu_\phi(t)\bigg\}
\end{equation*}
with
\begin{equation*}
	d\nu_\phi(t)\coloneqq \frac{\mathrm{dim}(\phi)}{\pi^{n-\frac{1}{2}}\Gamma(n-\frac{1}{2})}{\displaystyle\prod\limits_{i=1}^{n-1}(t^2+(a_i'+n-i-1)^2)}dt.
\end{equation*}
\end{theorem}

\begin{proof}
	The statement is a direct consequence of formula \eqref{thistensor2} and Propositions \ref{thisprop} and \ref{thisprop'}.
\end{proof}

 \begin{appendices}
	\begin{center}\section*{APPENDIX}
		\end{center}

In this short Appendix, for the sake of completeness, we include the branching rules for (most of) the complementary series of $\mathrm{Spin}(m+1,1)$ to the parabolic $P$, as obtained in \cite{gangoshima}. For the branching rules of the remaining unitary irreducible representations of $\mathrm{Spin}(m+1,1)$, we refer the interested reader to \cite[Section 3]{gangoshima}.

Assume that $G\coloneqq \mathrm{Spin}(2n,1)$. 
Let 
\begin{equation*}
	\gamma=(a_1+n-\frac{3}{2},\ldots,a_{n-1}+\frac{1}{2},a_n)
\end{equation*}
with $a_1\geq a_2\geq\ldots\geq a_{n-1}\geq 0$, $a_1,\ldots,a_{n-1}$ being all integers, $a_n\in\mathbb{R}$, $\vert a_n\vert<n-\frac{1}{2}$, and $a_j=0$ for any $n-\vert a_n\vert-\frac{1}{2}<j\leq n-1$. Set
\begin{equation*}
	\mu\coloneqq (a_1,\ldots,a_{n-1})
\end{equation*}
and 
\begin{equation*}
	\nu\coloneqq a_n\lambda_0.
\end{equation*}
The representation
\begin{equation*}
	\pi(\gamma)\coloneqq \mathrm{Ind}_{MA\bar{N}}^G(V_{M,\mu}\otimes e^{\nu-\rho'}\otimes1_{\bar{N}}) 
\end{equation*}
is a complementary series representation of $G$.
According to \cite[Theorem 3.20]{gangoshima},
\begin{equation*}
	\pi(\gamma)\vert_{ P}\cong \bigoplus\limits_{\tau}\mathrm{Ind}_{M'N}^P(V_{M',\tau}\otimes e^{i\xi_0})
\end{equation*}
where $\tau=(c_1,\ldots,c_{n-1})$ runs over tuples such that
\begin{equation*}
	a_1\geq c_1\geq a_2\geq c_2\geq \ldots\geq a_{n-1}\geq \vert c_{n-1}\vert
\end{equation*}
and $c_i-a_1\in\mathbb{Z}$.

Assume that $G\coloneqq \mathrm{Spin}(2n-1,1)$. Let
\begin{equation*}
	\gamma=(a_1+n-2,\ldots, a_{n-1},a_n)
\end{equation*} 
with $a_1\geq a_2\geq\ldots\geq a_{n-1}\geq 0$, $a_1,\ldots,a_{n-1}$ being all integers, $a_n\in\mathbb{R}$, $\vert a_n\vert<n-1$, and $a_j=0$ for any $n-\vert a_n\vert-1<j\leq n-1$. Set
\begin{equation*}
	\mu\coloneqq (a_1,\ldots,a_{n-1})
\end{equation*}
and 
\begin{equation*}
	\nu\coloneqq a_n\lambda_0.
\end{equation*}
The representation 
\begin{equation*}
	\pi(\gamma)\coloneqq \mathrm{Ind}_{MA\bar{N}}^G(V_{M,\mu}\otimes e^{\nu-\rho'}\otimes1_{\bar{N}})
\end{equation*}
is a complementary series representation. According to \cite[Theorem 3.23]{gangoshima},
\begin{equation*}
	\pi(\gamma)\vert_{ P}\cong \bigoplus\limits_{\tau}\mathrm{Ind}_{M'N}^P(V_{M',\tau}\otimes e^{i\xi_0}),
\end{equation*}
where $\tau=(c_1,\ldots,c_{n-2})$ runs over tuples such that 
\begin{equation*}
	a_1\geq c_1\geq a_2\geq c_2\geq \ldots\geq c_{n-2}\geq  \vert a_{n-1}\vert 
\end{equation*}
and $c_i-a_1\in\mathbb{Z}$.

\end{appendices}

\end{document}